\documentclass[12pt,a4paper,reqno]{amsart}
\usepackage{amsmath, amsthm, amsfonts, amssymb}
\usepackage{enumerate}
\usepackage{hyperref}
\usepackage[dvips]{graphicx}

\numberwithin{equation}{section}
\theoremstyle{plain}
\newtheorem{theorem}{Theorem}[section]
\newtheorem{lemma}[theorem]{Lemma}
\newtheorem{definition}[theorem]{Definition}

\newtheorem{proposition}[theorem]{Proposition}

\theoremstyle{definition}

\def\E{\mathbb{E}}
\def\Z{\mathbb{Z}}
\def\R{\mathbb{R}}
\def\T{\mathbb{T}}
\def\C{\mathbb{C}}
\def\N{\mathbb{N}}

\def\F{\mathbb{F}}

\def\transpose{\top}

\newcommand{\ud}{\,\mathrm{d}}

\newcommand{\Zmod}[1]{\Z_{#1}} 
\providecommand{\abs}[1]{\lvert#1\rvert}
\providecommand{\norm}[1]{\lVert#1\rVert}

\begin{document}
\title{A removal lemma for linear configurations in subsets of the circle}
\author{Pablo Candela}
\address{Centre for Mathematical Sciences\\
	Wilberforce Road\\
	Cambridge CB3 0WB\\
	United Kingdom}
\email{pc308@cam.ac.uk}

\author{Olof Sisask}
\address{School of Mathematical Sciences\\
	Queen Mary, University of London\\
	Mile End Road\\
	London E1 4NS\\
	United Kingdom}
\email{O.Sisask@qmul.ac.uk}

\thanks{Both authors are EPSRC postdoctoral fellows and gratefully acknowledge the support of the EPSRC}

\begin{abstract}
We obtain a removal lemma for systems of linear equations over the circle group, using a similar result for finite fields due to Kr\'al, Serra and Vena, and we discuss some applications.
\end{abstract}
\maketitle
\section{Introduction}
If a subset of an abelian group contains very few linear configurations of some given type, then one needs to delete only a few elements from the set in order to remove all such configurations. This is the moral of so-called \emph{arithmetic removal lemmas}. For example, if $A$ is a subset of a cyclic group $\Z_N=\Z/N\Z$ containing only $\delta N^2$ of its own sums (i.e. solutions to $a_1+a_2=a_3$), then one can make $A$ completely sum-free by deleting only $\delta' N$ of its elements, where $\delta'$ depends only on $\delta$, and $\delta' \to 0$ as $\delta\to 0$. In \cite{GAR} Green proved a result of this type dealing with the removal of solutions to a single linear equation over an arbitrary finite abelian group. Green raised the question of whether similar results held for systems of equations, noting that the Fourier analytic methods employed in \cite{GAR} did not extend to give this. Shapira \cite{Shap} and (independently) Kr\'al, Serra and Vena \cite{KSVGC} used hypergraph removal results to obtain the following extension, dealing with systems of linear equations over finite fields:
\begin{theorem}\label{KSVGC}
Let $r\leq m$ be positive integers and let $\epsilon>0$. There exists $\delta>0$ such that the following holds. Let $\F$ be the finite field of order $q$, let $L$ be an $r\times m$ matrix with coefficients in $\F$ of rank $r$ over $\F$, and suppose $A_1,\ldots,A_m \subset\F$ satisfy $\E_{x\in \ker L}\, 1_{A_1}(x_1)\cdots 1_{A_m}(x_m)\leq \delta$. Then there are sets $E_1,\ldots,E_m\subset \F$ of cardinality at most $\epsilon q$ such that $(A_1\setminus E_1)\times \cdots \times (A_m\setminus E_m)\cap \ker L=\emptyset$.
\end{theorem}
Our aim here is to obtain a continuous analogue of Theorem \ref{KSVGC}, replacing finite fields with the circle group $\T = \R/\Z$. Previous extensions of discrete additive-combinatorial results to the latter setting include the analogues of the Cauchy-Davenport inequality obtained by Raikov \cite{Raikov} and Macbeath \cite{Macbeath}---see the excellent notes \cite{ruzsa:sumsets_structure} of Ruzsa for a more detailed account of this topic---and Lev's work \cite{lev:continuous-sum-free} on sum-free sets in $\T$.

To state our main result let us set up some notation. For any compact abelian group $G$ we denote the normalized Haar measure on $G$ by $\mu_G$. We denote the closed subgroup $\{x\in G^m:Lx=0\}$ of the direct product $G^m$ by $\ker_G L$, and to abbreviate the notation we denote by $\mu_L$ the normalized Haar measure on $\ker_G L$. For measurable functions $f_1,f_2,\ldots,f_m:G\to\C$ we define
\begin{equation}\label{SolMeas}
S_L(f_1,\ldots,f_m)=\int_{\ker_G L} f_1(x_1)\cdots f_m (x_m)\ud\mu_L(x).
\end{equation}
(Throughout the paper ``measurable" refers to Borel measurability.)
If each $f_i$ is the indicator $1_{A_i}$ of a measurable set $A_i\subset \T$, then \eqref{SolMeas} becomes simply $S_L(A_1,\ldots,A_m)=\mu_L(A_1\times \cdots \times A_m\cap \ker_G L)$. We refer to the latter quantity as the \emph{solution measure} of the sets $A_i$. When $A_i=A$ for all $i\in [m]=\{1,2,\ldots,m\}$, we write $S_L(A)$ for the solution measure.
If the group $G$ has to be specified to avoid confusion, we shall write $\mu_{L,G},S_{L,G}$ instead of $\mu_L,S_L$.
The main result, then, is the following.
\begin{theorem}\label{T-system-removal}
Let $L$ be an $r\times m$ matrix of integers, of full rank $r$. For any $\epsilon>0$, there exists $\delta=\delta(L,\epsilon)>0$ such that the following holds. If $A_1,\ldots,A_m$ are measurable subsets of $\T$ such that $S_L(A_1,\ldots,A_m)\leq \delta$, then there are measurable sets $E_1,\ldots,E_m\subset \T$ with $\mu_\T(E_i)\leq \epsilon$ for all $i\in [m]$, such that $(A_1\setminus E_1)\times \cdots \times (A_m\setminus E_m)\cap \ker_\T L=\emptyset$.
\end{theorem}
For completeness we prove also the following variant concerning sets with zero solution-measure, which has a much simpler proof.
\begin{proposition}\label{0-removal}
Let $L$ be an $r\times m$ matrix of integers, of full rank $r$, and suppose $A_1,\ldots,A_m$ are measurable subsets of $\T$ such that $S_L(A_1,\ldots,A_m)=0$. Then there are null sets $E_1,\ldots,E_m\subset \T$ such that $(A_1\setminus E_1)\times \cdots \times (A_m\setminus E_m)\cap \ker_\T L=\emptyset$. We can take $A_i\setminus E_i$ to be the set of Lebesgue density points of $A_i$.
\end{proposition}

We now discuss briefly some consequences of these results. We say an integer matrix $L$ is  \emph{invariant} if it satisfies $L\mathbf{1}=0$ for the constant vector $\mathbf{1}=(1,1,\ldots,1)$. In this case the system $Lx=0$ is translation invariant in the sense that given any abelian group $G$, for any $x=(x_1,\ldots,x_m)\in G^m$ and $t\in G$, we have $Lx=0$ if and only if $L(x_1+t,\ldots,x_m+t)=0$. In particular, for any $t\in G$ the element $x=(t,\ldots,t)$ is a solution of the system. Therefore, Proposition \ref{0-removal} implies that if $L$ is invariant then any set $A\subset \T$ of positive measure has $S_L(A)>0$. However, the latter positive quantity may depend on the set $A$. By contrast, Theorem \ref{T-system-removal} implies the following analogue of Szemer\'edi's theorem \cite[Theorem 11.1]{T-V} for translation-invariant systems on $\T$.
\begin{theorem}\label{T-Szem}
Let $L$ be an invariant $r\times m$ integer-matrix of full rank $r$. Then for any $\alpha>0$, there exists $c=c(\alpha,L)>0$ such that for any measurable set $A\subset \T$ of measure at least $\alpha$, we have $S_L(A)\geq c$.
\end{theorem}
For instance, since arithmetic progressions of arbitrary fixed length are translation invariant, any subset of the circle of positive measure $\alpha$ contains a positive measure $c$ of such progressions, where $c$ depends on $\alpha$ but not on the particular subset.

At the end of the paper we discuss another application of Theorem \ref{T-system-removal}, related to the role that groups such as the circle can play as limit objects for certain additive-combinatorial problems.

The paper has the following outline.
Our proof of Theorem \ref{T-system-removal} reduces the problem to the discrete case, where one can appeal to Theorem \ref{KSVGC}. This involves first approximating each set $A_i$ by a simpler set that can be viewed as a subset $A_i'$ of a cyclic group $\Z_p$ for $p$ a prime. This is done in Section \ref{section:discretization}. The relationship between the solution-measure of the approximating sets and the solution-counts on $\Z_p$ of the sets $A_i'$ is captured in Lemma \ref{sol-measure-relation}. This relationship is somewhat subtle, in that expressing the solution-measure in terms of the latter discrete solution-counts requires many different shifts of the set $A_1' \times \cdots \times A_m'$, each shift having a  corresponding weight. We then require some control on these weights, which is obtained in Section \ref{section:weights} using a simple geometric characterization of $\ker_\T L$ and its measure $\mu_L$. The proof of Theorem \ref{T-system-removal} is then completed in Section \ref{section:main_proofs}, where we also deduce Theorem \ref{T-Szem} and prove Proposition \ref{0-removal}. Finally we close with the above-mentioned application and some further remarks in Section \ref{section:remarks}.

\section{A discrete decomposition of the solution measure}\label{section:discretization}

\subsection{Approximating measurable sets}
For any positive integer $N$, we refer to the partition $\T=\bigsqcup_{x\in [N]}\, [(x-1)/N,x/N)$ as the $N$\emph{-partition of} $\T$, and we say $A\subset \T$ is $N$\emph{-measurable} if $A$ is a union of intervals from the $N$-partition. The aim in this subsection is to show that, for the  proof of Theorem \ref{T-system-removal}, the sets $A_i$ can be assumed to be $p$-measurable for some large prime $p$.

\begin{lemma}\label{approx}
Let $L$ be an $r\times m$ matrix of integers, of full rank $r$, such that any $r\times (m-1)$ submatrix of $L$ also has rank $r$.
Let $\delta>0$ and let $C_1,\ldots, C_m$ be measurable subsets of $\T$. Then for any large $p\in \N$, there exist $p$-measurable sets $A_i\subset \T$ such that $\mu_\T(C_i\Delta A_i)\leq \delta/m$ for all $i\in [m]$, and $|S_L(C_1,\ldots,C_m)-S_L(A_1,\ldots,A_m)|\leq \delta$.
\end{lemma}

The submatrix condition in this lemma can be assumed without loss of generality when proving Theorem \ref{T-system-removal}. Indeed, suppose that deleting column $j$ from $L$ yields a matrix $L'$ of rank $r-1$. Then for some non-zero vector $v\in \Z^r$, we have $v^\transpose L'=0$. Since $L$ has rank $r$, the $j$th entry of $v^\transpose L$ must be a non-zero integer $\ell$. Now if $x\in A_1\times\cdots\times A_m$ satisfies $Lx=0$, then in $\T$ we have $\ell \, x_j=(v^\transpose L)\cdot x= v^\transpose \cdot (L x)=0$. Therefore we can delete all such solutions $x$ by removing the finite set $\{a\in A_j:\ell\,a=0\}$ from $A_j$, so Theorem \ref{T-system-removal} is clearly true for this system.

To prove Lemma \ref{approx} we use the following basic result, which will also be used later.
\begin{lemma}\label{basic}
Let $G$ be a locally compact abelian group with a Haar measure $\mu$ and let $H$ be a closed subgroup of $G^m$ with a Haar measure $\mu_H$ such that the projection $\pi : H \to G$, $x \mapsto x_i$ is surjective. Then there is a constant $c > 0$ such that for any functions $f_1,\ldots,f_m:G\to \C$ with $\|f_j\|_{L_\infty}\leq 1$ for all $j$, we have \[\Big| \int_H\; f_1(x_1)f_2(x_2)\cdots f_m(x_m) \ud \mu_H(x) \Big| \leq c \norm{f_i}_{L_1}.\]
If $G,H$ are compact abelian groups and $\mu_G,\mu_H$ are their respective unique probability Haar measures, then we can take $c = 1$.
\end{lemma}
\begin{proof}
The left side above is at most $\int_H \abs{f_i(x_i)} \ud\mu_H(x)=\int_H \abs{f_i\circ \pi(x)} \ud\mu_H(x)$. The map $\pi$ is a continuous surjective homomorphism from $H$ to $G$, whence the measure $\mu_H\circ \pi^{-1}$ is a Haar measure on $G$, so by uniqueness there exists $c>0$ such that $\mu_H\circ \pi^{-1}=c \mu_G$, and $c =1$ if $\mu_G,\mu_H$ are both probability measures. It follows that $\int_H \abs{f_i\circ \pi(x)} \ud\mu_H(x)=c \int_G \abs{f_i(y)}\ud\mu_G(y)=c \norm{f_i}_{L_1}$.
\end{proof}

\begin{proof}[Proof of Lemma \ref{approx}]
First, by basic measure theory there is some large integer $N=2^n$ such that each $C_i$ can be approximated within $\delta/2m$ by an $N$-measurable set $B_i$. Then, for any $p$ large enough in terms of $N$, we can approximate each $B_i$ by a $p$-measurable set $A_i$ with $\mu_\T(B_i\Delta A_i)\leq \delta/2m$, simply by taking $A_i$ to be the union of the intervals in the $p$-partition of $\T$ that are contained in $B_i$. Thus $\mu_\T(C_i \Delta A_i) \leq \delta/m$.

Now, by the multilinearity of $S_L$, we have
\begin{align*}
\abs{S_L(C_1,\ldots,C_m)-S_L(A_1,\ldots,A_m)} &\leq \sum_{i\in [m]} \abs{ S_L (1_{A_1},\ldots,1_{A_{i-1}},1_{C_i}-1_{A_i},1_{C_{i+1}},\ldots,1_{C_m}) },
\end{align*}
and the assumption that every $r \times (m-1)$ submatrix of $L$ has rank $r$ is easily seen to imply that each projection $\ker_\T L \to \T$, $x \mapsto x_i$ is surjective, whence by Lemma \ref{basic} the $i$th summand above is at most 
$\norm{1_{C_i} - 1_{A_i}}_{L_1(\T)} \leq \mu_\T(C_i\Delta A_i)\leq \delta/m$.
\end{proof}

\subsection{The main formula}
From now on, given a $p$-measurable set $A\subset\T$, we denote by $A'$ the subset of $\Z_p$ defined by $1_{A'}(x)=1_A(x/p)$. In order to apply Theorem \ref{KSVGC}, we express $S_L(A)$ in terms of solution measures in $\Z_p$ involving $A'$. This is done in Lemma \ref{sol-measure-relation} below.
 
For any positive integer $p$, let $\Lambda=\Lambda(p)$ denote the discrete torus $\Z_p^m/p\leq \T^m$, with elements denoted $j/p=\big(j(1)/p,\ldots,j(m)/p\big),\, j\in \Z_p^m$.
\begin{definition}
For any $r\times m$ integer matrix $L$ and any positive integer $p$, we define
\[J=J(L,p)=\{j/p\in\Lambda: \mu_L\big(\;(j/p+[0,1/p)^m)\;\cap \ker_\T L\big)>0\}.\]
\end{definition}
The fact that $J$ consists of $O_L(1)$ shifts of $\Lambda\cap\ker_\T L$ is central to the whole argument.
\begin{lemma}\label{J-cover}
For some $K_L>0$ depending only on $L$, for any $p$ there exist elements $j_1/p,j_2/p,\ldots,j_K/p\in J(L,p)$, $K\leq K_L$, such that $J = \bigsqcup_{k\in [K]}\big(j_k/p+(\Lambda\cap\ker_\T L)\big)$.
\end{lemma}
\begin{proof}
If $j/p\in J$, then $L(j/p)$ lies in $-\left( L([0,1)^m) \cap \Z^r \right)/p \mod 1$. The latter finite set has size bounded in terms of $L$ alone. Choosing $j_1/p,\ldots,j_K/p\in J$ such that $L$ is a bijection from $\{j_1/p,\ldots,j_K/p\}$ to $L(J)$, we then have $K = O_L(1)$, and the result follows since $J\subset L^{-1} (L(J))$ and $L^{-1}\big(L(j_k/p)\big)\cap \Lambda=j_k/p+(\Lambda\cap\ker_\T L)$.
\end{proof}
We can now prove the main formula.
\begin{lemma}\label{sol-measure-relation}
Let $L$ be an $r\times m$ matrix of integers of full rank $r$, let $p$ be a large prime, and let $A_1,\ldots, A_m$ be $p$-measurable subsets of $\T$. Then there exist $j_1/p,\ldots,j_K/p\in J$, with $K \leq K_L$, such that
\begin{equation}\label{sol-rel-formula}
S_{L,\T}(A_1,\ldots,A_m)=\sum_{k\in [K]} \lambda_k\; S_{L,\Z_p}\Big(A_1'-j_k(1),\ldots, A_m'-j_k(m)\Big),
\end{equation}
where $\lambda_k= p^{m-r} \mu_{L,\T}\Big( \big( j_k/p+[0,1/p)^m \big)\cap \ker_\T L \Big)$.
\end{lemma}
\begin{proof}
We have $S_{L,\T}(A_1,\ldots,A_m)$ equal to
\[
\mu_L\big((A_1\times\cdots\times A_m) \cap\ker_\T L\big)=\sum_{j\in A_1' \times \cdots \times A_m'}
\mu_L\Big( \big(j/p+[0,1/p)^m \big)\cap \ker_\T L \Big).
\]
By the definition of the set $J$, this sum can be restricted directly to the shifts $j_1/p+\ker_\T L,\ldots, j_K/p+\ker_\T L$ occurring in Lemma \ref{J-cover}; since the subgroup $\Lambda\cap \ker_\T L$ of $\T^m$ is clearly isomorphic to the subgroup $\ker_{\Z_p} L$ of $\Z_p^m$, we see that
$S_{L,\T}(A_1,\ldots,A_m)$ equals
\begin{eqnarray*}
&  & \sum_{k\in [K]}\;\; \sum_{j\in A_1'\times \cdots \times A_m' \cap (j_k+\ker_{\Z_p} L)}
\mu_L\Big( ( j/p+[0,1/p)^m )\cap \ker_\T L \Big)\\
& = & \sum_{k \in [K]}\;\; \sum_{j\in (A_1'\times \cdots \times A_m'-j_k) \cap \ker_{\Z_p} L}
\mu_L\Big( \big( (j+j_k)/p+[0,1/p)^m \big)\cap \ker_\T L \Big).
\end{eqnarray*}
By invariance of $\mu_L$ under translation by $j/p\in \ker_\T L$, this equals
\[
\sum_{k\in [K]}\;\; \sum_{j\in (A_1'\times \cdots \times A_m'-j_k) \cap \ker_{\Z_p} L}
\mu_L\Big( \big( j_k/p+[0,1/p)^m \big)\cap \ker_\T L \Big),
\]
and \eqref{sol-rel-formula} follows.
\end{proof}

\section{A positive lower bound for the weights $\lambda_k$}\label{section:weights}

For each $j\in \Z_p^m$, let $\lambda(j)=p^{m-r} \mu_L\Big( \big( j/p+[0,1/p)^m \big)\cap \ker_\T L \Big)$.\\
In order to use Lemma \ref{sol-measure-relation}, we require that the weights $\lambda_k$ be bounded away from 0, uniformly over $p$. Such a bound is guaranteed by the following result.
\begin{lemma}\label{estimate}
Let $L$ be an $r\times m$ matrix of integers of full rank $r$. Then there exists $\lambda^*>0$ depending only on $L$ such that, for any large positive integer $p$, for any $j/p\in J(L,p)$, we have $\lambda(j) \geq \lambda^*$.
\end{lemma}
The proof relies on a compactness argument coupled with the geometric characterization of $\mu_L$ given in Lemma \ref{Haar-Lebesgue} below. In what follows we always consider $\T^m$ as the set $[0,1)^m\subset \R^m$ with coordinate-wise addition modulo 1 (and with topology the quotient topology on $\R^m/\Z^m$). Then $\ker_\T L$ is the closed subgroup $\{x\in [0,1)^m: Lx\in \Z^r\}\leq \T^m$. This subgroup is described more precisely by the following simple result.

\begin{lemma}\label{ker-charac}
Let $x_1,\ldots,x_M$ be a choice of points in $[0,1)^m$ such that the linear map $L$ over $\R$ gives a bijection $\{x_i:i\in [M]\}\to L([0,1)^m)\cap\Z^r$. Then we have the partition
\begin{equation}\label{char}
\ker_\T L = \bigsqcup_{i\in [M]} \big((x_i+_\R\ker_\R L)\cap [0,1)^m\big).
\end{equation}
\end{lemma}

Here we use $+_\R$ to denote addition in $\R^m$ (or more generally addition over $\R$), to distinguish it from addition in $\T^m$, which we may denote by $+_\T$.
We now use \eqref{char} to relate the Haar measure $\mu_L$ to the $(m-r)$-dimensional Lebesgue measure on $\ker_\R L$, which we denote $\mu_{L,\R}$.

\begin{lemma}\label{Haar-Lebesgue}
For any Borel set $A\subset \ker_\T L$, let $A^{(i)}:=A\cap (x_i+_\R\ker_\R L)$ for each $i\in [M]$. Then there is a constant $c_L > 0$ such that
$\mu_{L,\T}(A)= c_L \sum_i \mu_{L,\R}(A^{(i)}-_\R x_i)$.
\end{lemma}
\begin{proof}
Let $G$ denote the group $\{x\in \R^m: Lx\in \Z^r\}$. This is a closed subgroup of $\R^m$, and $H :=\Z^m\leq G$. Clearly we may identify $\ker_\T L$ with $G/H$. Thus, in the notation from \eqref{char}, we have $G=(\sqcup_{i\in [M]} x_i+\ker_\R L)+\Z^m$, so we may write
$G = \bigsqcup_{z \in Z} (z + \ker_\R L)$, for some collection $Z \subset \bigcup_i\{x_i\} + \Z^m$ containing the $x_i$. It is then easy to verify that Haar measure on $G$ must be a multiple of
\begin{equation}\label{eqn:G_meas}
\mu_G(A) := \sum_{z \in Z} \mu_{L,\R}\left( A^{(z)} - z \right),
\end{equation}
where $A^{(z)} = A \cap (z + \ker_\R L)$, by considering its restriction to $\ker_\R L$. Endowing $H$ with counting measure, by the quotient integral formula \cite[Thm 1.5.2]{principlesHA} there is an invariant Radon measure $\mu_{G/H}\neq 0$ on $G/H$ such that
\begin{equation}\label{eqn:QIF}
\int_G f \ud \mu_G = \int_{G/H} \sum_{n \in \Z^m} f(x+_\R n) \ud \mu_{G/H}(x)
\end{equation}
for any $f \in L_1(G)$. By the uniqueness of Haar measure we have $\mu_{L,\T} = c_L \mu_{G/H}$ for some constant $c_L > 0$. Now, given a Borel subset $A$ of $\ker_\T L$, the function $f = 1_A$ on $G$ is integrable, and the function $\sum_{n \in \Z^m} f(x+n)$ on $G/H$ is simply $1_A$, whence by \eqref{eqn:QIF} we have $\mu_{G/H}(A)=\int_G 1_A(x) \ud\mu_G(x)$ and by \eqref{eqn:G_meas} this is $\sum_{i\in [M]}\mu_{L,\R} (A^{(i)}- x_i)$.
\end{proof}
Lemma \ref{estimate} follows immediately from the following result.
\begin{lemma}\label{reduct}
There exists a finite set $\Lambda^*$ of positive quantities, depending only on $L$, such that for all large positive integers $p$ we have $\{\lambda(j):j\in J(p,L)\}\subset \Lambda^*$.
\end{lemma}
\begin{proof}
First we show that there is a finite set $U\subset \ker_\T L$, depending only on $L$, such that for any large $p$ and $j/p\in J$ there exist $v\in \Z^m$ and $u\in U$ such that
\begin{equation}\label{uv}
\lambda(j)=c_L\; \mu_{L,\R}\big((u+_\R v+_\R [0,1)^m)\cap \ker_\R L\big)>0,
\end{equation}
where $c_L$ is the constant from Lemma \ref{Haar-Lebesgue}. For $p$ large enough depending only on $L$, by \eqref{char} the set $(j/p+[0,1/p)^m )\cap\ker_\T L$ lies entirely in $x_i+_\R\ker_\R L$ for some $i\in [M]$, so
\begin{eqnarray*}
\lambda(j)&=&p^{m-r}\mu_L\big(\,(j/p+[0,1/p)^m)\cap\ker_\T L\big)\\
&=&c_L\;p^{m-r} \mu_{L,\R} \big(\,(j/p-_\R x_i+_\R [0,1/p)^m)\cap\ker_\R L\big)\\
&=&c_L\; \mu_{L,\R} \big(\,(j-_\R p\, x_i+_\R[0,1)^m)\cap\ker_\R L\big),
\end{eqnarray*}
where $j\in \Z^m$. Now let $U=\bigcup_{i\in [M]}\{-p\, x_i \mod 1:p\in \N\}\subset \ker_\T L$. This is a finite subset of $\ker_\T L$ if we take the $x_i$ to have rational coordinates (as we do). For any $j/p\in J$, we then have $j-_\R p\, x_i=u+_\R v$ for some $u\in U$ and $v\in \Z^m$, whence \eqref{uv} follows.

Now, by translation invariance of $\mu_{L,\R}$ by elements of $\ker_\R L$, the measure in \eqref{uv} depends only on $L(u+v)$. But if this measure is positive, then $L(u+v)$ is contained in the finite set $L\left(\bigcup_{w\in U} w+\Z^m \right) \cap -L\left([0,1)^m\right)$. Hence there are only finitely many possible values for the left-hand side of \eqref{uv}.
\end{proof}

\section{Proofs of the main results}\label{section:main_proofs}

Recall that whenever $A$ is a $p$-measurable subset of $\T$ we denote by $A'$ the corresponding subset of $\Z_p$ defined by $1_{A'}(x)=1_A(x/p)$.
\begin{proof}[Proof of Theorem \ref{T-system-removal}]
Given the matrix $L$, fix $\epsilon>0$, let $\lambda^*>0$ be the lower bound given by Lemma \ref{estimate}, and let $K_L>0$ be as defined in
Lemma \ref{J-cover}. Let $\delta'>0$ be such that Theorem \ref{KSVGC} holds with initial parameter $\epsilon/2K_L$, and let $\delta=\min(\delta'\lambda^*,\epsilon)/2$. Now let $A_i\subset \T$, $i\in[m]$, be any Borel sets satisfying $S_L(A_1,\ldots,A_m)\leq \delta$. Applying Lemma \ref{approx}, we can assume that the given sets $A_i$ are $p$-measurable for some large prime $p$, up to an error of measure $\delta/m\leq\epsilon/2$ for each set, and such that $S_L(A_1,\ldots,A_m)\leq 2\delta\leq \delta'\lambda^*$. It follows from \eqref{sol-rel-formula} and the lower bound $\lambda_k\geq\lambda^*$ that for some $K\leq K_L$ and each $k\in [K]$, we have $S_{L,\Zmod{p}}(A_1'-j_k(1),\ldots,A_m'-j_k(m))\leq \delta'$, and so Theorem \ref{KSVGC} gives us subsets $E_{k,1},\ldots, E_{k,m}$ of $\Z_p$ of cardinality at most $\epsilon p/2K_L$ such that
\begin{equation}\label{nosol}
\big((A_1'\setminus E_{k,1})-j_k(1)\big)\times\cdots\times \big((A_m'\setminus E_{k,m})-j_k(m)\big) \cap \ker_{\Z_p} L=\emptyset.
\end{equation}
Now for each $i\in [m]$, define the $p$-measurable set $E_i=\bigcup_{k\in [K]}(E_{k,i}/p+[0,1/p))$, and note that $\mu_\T(E_i)\leq \epsilon/2$. Finally, for each $i\in [m]$ let $\Delta_i$ be the null set $\Z_p/p$ in $\T$. We now claim that
\[\prod_{i\in [m]} A_i\setminus (E_i\cup \Delta_i)\;\cap\; \ker_\T L=\emptyset.\]
Suppose for a contradiction that this set is non-empty, containing some point $x$. Then by the $p$-measurability of the sets $A_i\setminus E_i$ and the definition of $\Delta_i$, letting $j$ denote the point $(\lfloor p\,x_1\rfloor,\ldots,\lfloor p\,x_m\rfloor)\in \Z_p^m$, we have
\[\prod_{i\in [m]} A_i\setminus (E_i\cup \Delta_i)\supset j/p+(0,1/p)^m \ni x.\]
But then $\big(j/p+(0,1/p)^m\big)\cap\ker_\T L$ is a non-empty open subset of $\ker_\T L$, so this set must have positive $\mu_L$-measure, and so $j/p\in J$. Then, by the covering of $J$ in Lemma \ref{J-cover}, there exists $k\in [K]$ such that $j\in j_k+\ker_{\Z_p} L$, and so $j-j_k$ belongs to $\prod_i\big(\,(A_i'\setminus E_i')-j_k(i)\,\big) \cap\ker_{\Z_p} L$, contradicting \eqref{nosol}.
\end{proof}

We can now quickly deduce Theorem \ref{T-Szem}. We say $A\subset \T$ is $L$\emph{-free} if $A^m\cap \ker_\T L=\emptyset$.
\begin{proof}[Proof of Theorem \ref{T-Szem}]
Let $c$ be a positive value of $\delta$ such that Theorem \ref{T-system-removal} holds with initial parameter $\epsilon=\alpha/2m$. Suppose $S_L(A)\leq c$. Then by Theorem \ref{T-system-removal} there exists a measurable set $E\subset A$ such that $A\setminus E$ is $L$-free and
$\mu_\T(E)\leq\mu_\T(E_1)+\cdots+\mu_\T(E_m)\leq \alpha/2$. Since for any $a\in A\setminus E$ the constant element $(a,\ldots,a)\in \T^m$ is in $\ker_\T L$, we must have $A\setminus E=\emptyset$, and therefore $\mu_\T(A)=\mu_\T(E) < \alpha$.
\end{proof}
While Theorem \ref{T-Szem} follows very easily from Theorem \ref{T-system-removal}, one can in fact simplify the overall argument somewhat if one is only interested in the former theorem---see the first remark in the next section.
\begin{proof}[Proof of Proposition \ref{0-removal}]
For each $i\in [m]$ let $D_i$ denote the set of Lebesgue density points of $A_i$. Suppose for a contradiction that there exists some point $x$ in $D_1\times\cdots\times D_m\cap \ker_\T L$, and fix $\epsilon>0$. By the Lebesgue density theorem, there exists $\delta>0$ such that, letting $Q$ denote the cube centered on $x$ and of side-length $\delta$, we have
$\mu_\T(D_i\cap \pi_i Q)\geq (1-\epsilon)\delta$ for all $i$ (where $\pi_i$ denotes projection to the $i$th component on $\T^m$). Now, by Lemma \ref{ker-charac}, and the characterization of $\mu_{L,\T}$, setting $C_i:=D_i\cap \pi_i Q$ for each $i$, there exists a constant $c_L>0$ such that $\mu_L(C_1\times\cdots\times C_m \cap \ker_\T L) \geq c_L \delta^{m-r}\mu_{L,\R}(B_1\times\cdots\times B_m\cap \ker_\R L)$, 
where $B_i\subset [-1/2,1/2)$ is the dilation by $\delta^{-1}$ of the set $B_i'-x_i$, when the latter is viewed as a subset of $I := [-1/2,1/2]\subset \R$. We claim that the large density of each $B_i$ inside $I$ implies $\mu_{L,\R}(B_1\times\cdots\times B_m\cap \ker_\R L)>0$, which gives a contradiction. Indeed, by multilinearity and Lemma \ref{basic} we have that $\abs{ \mu_{L,\R}(I^m\cap \ker_\R L)-\mu_{L,\R}(B_1\times\cdots\times B_m\cap \ker_\R L) }$ is at most
\begin{gather*}
\sum_{i\in [m]} \abs{ \int_{\ker_\R L} 1_{B_1}(x_1)\cdots 1_{B_{i-1}}(x_{i-1})\;(1_I - 1_{B_i})(x_i)\;1_I(x_{i+1})\cdots 1_I(x_m) \ud\mu_{L,\R}(x) } \\
\leq c \sum_{i \in [m]} \norm{ 1_I - 1_{B_i} }_{L_1(\R)} \leq cm \epsilon.
\end{gather*}
Setting $\epsilon = \mu_{L,\R}(I^m \cap \ker_\R L)/2cm$ yields the claim. Note that the measure here is strictly positive since $I^m \cap \ker_\R L$ contains a non-empty open set. (In fact $\mu_{L,\R}(I^m \cap \ker_\R L) \geq 1$ by Vaaler's theorem \cite{Va}.)
\end{proof}
\section{Remarks}\label{section:remarks}
The precision of Lemma \ref{sol-measure-relation} is not required for a proof of Theorem \ref{T-Szem} per se; one can do with a simpler inequality of the form $S_{L,\T}(A_1,\ldots,A_m) \gg_L S_{L,\Zmod{p}}(A_1',\ldots,A_m')$. (If $L$ is invariant one can also apply Vaaler's theorem to obtain the more precise inequality $S_{L,\T}(A_1,\ldots,A_m) \geq S_{L,\Zmod{p}}(A_1',\ldots,A_m')$ for $p$-measurable sets $A_i$.) 
On the other hand, the non-trivial shifts of $A_1' \times \cdots \times A_m'$ that contribute to $S_L(A_1,\ldots,A_m)$ in Lemma \ref{sol-measure-relation} need to be taken into account when removing solutions from $A_1 \times \cdots \times A_m$ as in Theorem \ref{T-system-removal}.

As mentioned in the introduction, Theorem \ref{T-system-removal} can be used when studying $\T$ as a limit object or model for certain finite additive-combinatorial questions. A well-known question of this kind asks for the maximal density $d_L(\Z_p)$ of a subset of $\Z_p$ not containing solutions to a given system $Lx=0$. In \cite{candela-sisask}, the special case of Theorem \ref{T-system-removal} for a single equation was used to show that if $L$ is a linear form with integer coefficients in at least 3 variables then $d_L(\Z_p)$ converges to the natural analogue $d_L(\T):=\sup\{\mu_\T(A):A\subset \T\textrm{ is }L\textrm{-free}\}$ as $p\to\infty$ through the primes. Theorem \ref{T-system-removal} enables us to extend this convergence result to so-called systems of \emph{complexity} 1. A notion of complexity for systems of linear forms on finite abelian groups was introduced in the paper \cite{GWcomp}, to which we refer the reader for more background on this topic. We use the following variant of this notion, specific to groups $\Z_p$ and $\T$.
\begin{definition}
Let $L$ be an $r\times m$ integer matrix. We say the system of equations $Lx=0$ \textup{(}alternatively, the matrix $L$\textup{)} has \emph{complexity} $k$ if $k$ is the smallest integer such that, for any $\epsilon>0$, there exists $\delta>0$ with the following property: let $G=\T$ or $\Z_p$ for any large prime $p>p_0(L)$; then for any $f,g: G\to \C$ with $\|f\|_{L_\infty(G)},\|g\|_{L_\infty(G)}$ both at most 1 and $\|f-g\|_{U^{k+1}(G)}\leq \delta$, we have $|S_{L,G}(f)-S_{L,G}(g)|\leq \epsilon$.
\end{definition}
Here the notation $\|f\|_{U^k(G)}$ refers to the $k$th Gowers uniformity norm, which is defined on $L_\infty(G)$ for any compact abelian group $G$ \cite{ET}. Using Theorem \ref{T-system-removal}, the main convergence result from \cite{candela-sisask} can be extended as follows.
\begin{theorem}\label{CSgen}
Let $\mathcal{F}$ be a finite family of full-rank integer-matrices of complexity 1, and let $d_\mathcal{F}(\Z_p)$ denote the maximal density of an
$\mathcal{F}$-free subset of $\Z_p$. Then $d_\mathcal{F}(\Z_p)\to d_\mathcal{F}(\T)$ as $p\to\infty$ over primes.
\end{theorem}
Here $d_\mathcal{F}(\T):=\sup\{\mu_\T(A):A\subset\T\textrm{ is }\mathcal{F}\textrm{-free}\}$, where we say a measurable set $A\subset \T$ is $\mathcal{F}$-free if $A$ is $L$-free for every $L\in\mathcal{F}$. Generalizing the argument in \cite{candela-sisask} to obtain Theorem \ref{CSgen} is not hard; we omit the details in this paper.

Let us close with remarks regarding further generalizations of removal lemmas. Recently, Kr\'al, Serra and Vena extended Theorem \ref{KSVGC} to all finite abelian groups \cite{KSVGR}, and upon inspection Green's proof \cite{GAR} for single equations can be seen to hold over arbitrary compact abelian groups. Can Theorem \ref{T-system-removal} be generalized to all compact abelian groups? The desired generalization should hold with a function $\delta(L,\epsilon)$ independent of the group, so in particular $\delta$ should not depend on the group's topological dimension. The argument in this paper, when applied with $\T^n$ instead of $\T$, gives a parameter $\delta$ which decays to 0 as $n$ grows, so additional ideas are required.\\

\textbf{Acknowledgements.} The authors would like to thank Tim Austin for helpful conversations. Parts of this work were carried out while the authors attended the Discrete Analysis programme at the Isaac Newton Institute, whose support is gratefully acknowledged.

\end{document}